\documentclass[10pt, a4paper]{article}
\usepackage{amsmath, amsfonts, amssymb, amscd, amsthm, mathdots, mathrsfs, dsfont}
\usepackage{titlesec}
\usepackage{fancyhdr}
\usepackage{fullpage}
\usepackage[colorlinks, urlcolor=black]{hyperref}
\hypersetup{citecolor=blue, linkcolor=blue}
\usepackage{mathabx}

\newcommand{\A}{\mathcal{A}}
\newcommand{\B}{\mathcal{B}}
\newcommand{\BB}{\mathbb{B}}
\newcommand{\CC}{\mathbb{C}}
\newcommand{\DD}{\mathbb{D}}

\newcommand{\GL}{\mathrm{GL}}
\newcommand{\Hi}{\mathbf{H}}
\newcommand{\HH}{\mathbb{H}}

\newcommand{\ii}{\mathbf{i}}
\newcommand{\jj}{\mathbf{j}}

\newcommand{\M}{\mathcal{M}}
\newcommand{\NN}{\mathbb{N}}
\newcommand{\Ord}{\mathcal{O}}

\newcommand{\OrdI}{\mathcal{O}_\mathfrak{I}}

\newcommand{\OrdO}{\Ord_\mathfrak{O}}

\newcommand{\PSL}{\mathrm{PSL}}
\newcommand{\PSLBG}{\mathrm{PSL}_2^{\mathrm{BG}}}
\newcommand{\QQ}{\mathbb{Q}}

\newcommand{\q}{\mathbf{q}}
\newcommand{\RR}{\mathbb{R}}
\newcommand{\SL}{\mathrm{SL}}

\newcommand{\T}{\mathcal{T}}
\newcommand{\uu}{\mathbf{u}}
\newcommand{\U}{\mathcal{U}}
\newcommand{\ZZ}{\mathbb{Z}}

\DeclareMathOperator{\Hur}{Hur}
\DeclareMathOperator{\Mat}{M}
\DeclareMathOperator{\nrd}{nrd}

\DeclareMathOperator{\Stab}{Stab}
\DeclareMathOperator{\SO}{SO}
\DeclareMathOperator{\trd}{trd}

\titleformat{\section}[hang]
{\normalfont\filright\large}{\thesection . }{0pt}
{\upshape\bfseries}

\titleformat{\subsection}[hang]
{\itshape}{\thesubsection \ - }{0pt}
{}

\newtheorem{theorem}{Theorem}[section]
\newtheorem{proposition}[theorem]{Proposition}

\newtheorem{lemma}[theorem]{Lemma}

\theoremstyle{definition}
\newtheorem{definition}{Definition}[section]

\theoremstyle{remark}
\newtheorem{remark}{Remark}
\newtheorem{example}{Example}

\title{ Bianchi and Hilbert-Blumenthal quaternionic orbifolds}
\author{\small ALBERTO VERJOVSKY\footnote{Instituto de Matemáticas, Unidad Cuernavaca, Universidad Nacional Autóoma de México.  Av. Universidad s/n. Col. Lomas de Chamilpa, CP 62210, Cuernavaca, Mexico. \emph{E-mail address: }\texttt{alberto@matcuer.unam.mx}}  \; \&
ADRI$\acute{\mbox{A}}$N ZENTENO \footnote{Centro de Investigación en Matemáticas, A.C. Jalisco s/n. Col. Valenciana, CP 36023, Guanajuato, Mexico. \emph{E-mail address: }\texttt{adrian.zenteno@cimat.mx}}}
\date{\today}
\begin{document}

\maketitle

\begin{abstract}

In a series of papers, published in {\it Mathematische Annalen}, Bianchi and Blumenthal introduced the notions of Bianchi orbifolds and Hilbert-Blumnethal surfaces as generalizations of modular curves associated to quadratic fields. In this paper, in the same spirit, and following a similar line of reasoning, we introduce the concept of Bianchi and Hilbert-Blumenthal quaternionic orbifolds as generalizations of the Lipschitz and Hurwitz quaternionic modular orbifolds defined recently by Díaz, Vlacci and the first author. 
In particular, we describe the cusp cross-sections of the Hilbert-Blumenthal quaternionic orbifolds in terms of fundamental units of real quadratic fields. These are 7-dimensional solvmanifolds which are virtual ${\mathbb T}^6$ bundles over the circle with monodromy a linear Anosov diffeomorphism of the 6-torus.

2020 \textit{Mathematics Subject Classification}. 11R52, 20H10, 11F06.
\end{abstract}

\section{Introduction}\label{sec:intro}

The classical modular curve is defined as the quotient $\PSL_2(\ZZ) \backslash \Hi^1_{\CC}$ of the hyperbolic complex upper-half plane $\Hi^{1}_{\CC}$ by the action of the modular group $\PSL_2(\ZZ)$. 
A way  to get the modular group using quaternion algebras is as follows: First, we start with the indefinite quaternion algebra $\big(\frac{1,1}{\QQ}\big)$, which is isomorphic to $\mbox{M}_2(\QQ)$. Then, we look for the integral elements inside $\mbox{M}_2(\QQ)$, to get the $\ZZ$-order $\mbox{M}_2(\ZZ)$, and take its units group $\GL_2(\ZZ)$. Finally, we consider the index two subgroup $\GL^+_2(\ZZ) = \SL_2(\ZZ)$ of elements of $\GL_2(\ZZ)$ with determinant 1 and projectivize it, obtaining the Fuchsian group $\PSL_2(\ZZ) \subseteq \PSL_2(\RR)$ which acts faithfully on $\Hi^{1}_{\CC}$ by orientation-preserving isometries. In a similar way, the Bianchi (resp. Hilbert-Blumenthal) modular groups used to construct the corresponding Bianchi orbifolds (resp. Hilbert-Blumnethal surfaces), can be obtained by considering indefinite quaternion algebras of the form $\big(\frac{1,1}{K}\big)$ where $K$ is an imaginary (resp. real) quadratic field. These groups are Kleinian (resp. 2-fold Fuchsian) groups acting on the hyperbolic real 3-space $\Hi^{3}_{\RR}$ (resp. $\Hi^{2}_{\CC} := \Hi^{1}_{\CC} \times \Hi^{1}_{\CC}$) by orientation-preserving isometries. See \cite[Chapter 38]{Voight2021} for details.

On the other hand, by looking at the definite quaternion algebra $\big(\frac{-1,-1}{\QQ}\big)$, Díaz, Vlacci and the first author \cite{DiazVerjovskyVlacci2015} have recently introduced the Lipschitz and Hurwitz modular groups $\Gamma \subseteq \PSL_2(\HH)$ as generalizations of the classical modular group. 
These new quaternionic modular groups are characterized by certain conditions introduced by Bisi and Gentili in \cite{BisiGentili2008}, which allow us to obtain a proper and discontinuous action on the hyperbolic quaternionic half space $\Hi^1_{\HH}:=\{ \q \in \HH: \Re(z)>0 \}$ and define the quaternionic orbifolds $\Gamma \backslash \Hi^1_{\HH}$. 
We remark that Bisi-Gentili conditions  are a variation of certain conditions described by Ahlfors in \cite{Ahlfors1985}. See \cite{Sheydvasser2021} for the definition of orbifolds using Ahlfors conditions. 

In this paper, following the line of reasoning of Bianchi \cite{Bianchi1891} \cite{Bianchi1892} and Blumenthal \cite{Blumenthal1903} \cite{Blumenthal1904}, we define Bianchi and Hilbert-Blumenthal quaternionic modular groups as generalizations of Lipschitz and Hurwitz quaternionic modular group. On the one hand, considering a $\ZZ$-order $\Ord$ in the definite quaternion algebra $\big(\frac{-1,-1}{\QQ}\big)$, we define the Bianchi quaternionic modular group $\PSL_2(\Ord)$ which acts conformally on the Hamilton's quaternions $\HH$. This action can be extended to an action on the real hyperbolic 5-space $\Hi^5_{\RR}$ by Poincaré Extension Theorem as in \cite{DiazVerjovskyVlacci2015}\footnote{An extended version available in arXiv:1503.07214}, which allows us to define the Bianchi quaternionic orbifold 
\[
M_\Ord := \PSL_2(\Ord) \backslash \Hi^5_{\RR}.
\] 
On the other hand, let $K$ be a quadratic real field with ring of integers $\ZZ_K$. For every $\ZZ_K$-order $\Ord$ in the totally definite quaternion algebra $\big(\frac{-1,-1}{K}\big)$, we define the Hilbert-Blumenthal quaternionic modular group $\Gamma(\Ord) \subseteq \PSLBG(\Ord)$ associated to $\Ord$ which acts on the product $\Hi^{2}_{\HH} := \Hi^{1}_{\HH} \times \Hi^{1}_{\HH}$ of two hyperbolic quaternionic half spaces, producing the Hilbert-Blumenthal quaternionic orbifold 
\[
M_{\Gamma(\Ord)} : = \Gamma(\Ord) \backslash \Hi^2_{\HH}.
\]

Finally, it is well-known that a cusp cross-section of a Hilbert-Blumenthal surface is the 3-dimensional mapping torus of some Anosov diffeomorphism of the torus  \cite{Hirzebruch1973}. At the end of this paper, we give a description of the cusp at $\infty$ of the Hilbert-Blumenthal quaternionic orbifold $M_{\Gamma (\Ord)}$. In particular, we prove that a cross-section of the cusp at $\infty$ is a virtual $(6,1)$-torus bundle described in terms of the fundamental unit $\varepsilon$ of $\ZZ_K$. 

\subsection*{Acknowledgments} The first author was supported by grant  IN108120, PAPIIT, DGAPA, Universidad Nacional Aut\'onoma de M\'exico. The second author was supported by the CONACYT grant 432521, Estancias Posdoctorales por México 2021 - Modalidad Académica.


\section{Preliminaries on quaternionic Möbius transformations}\label{sec:2}

In this section, we will recall some definitions concerning quaternion algebras and review some classical facts about quaternionic Möbius transformation which we will use throughout this paper.  Our main references are \cite{Ahlfors1982}, \cite{Ahlfors1985}, \cite{BisiGentili2008}, \cite{DiazVerjovskyVlacci2015}, \cite{Vigneras1980} and \cite{Voight2021}. We refer the reader to loc. cit. for more details and references. 


\subsection{General definitions}\label{sec:21}

\begin{definition}
Let $K$ be a field of characteristic 0.  A $K$-algebra $\B$ is a \emph{quaternion algebra} over $K$, if there exist $\ii, \jj \in \B$ such that $\{ 1, \ii, \jj, \ii \jj \}$ is a $K$-basis for $\B$ and 
\[
\ii^2=a, \; \jj ^2=b \; \mbox{and}\; \ii \jj=- \jj \ii, 
\]
for some $a,b\in K^\times$. Such an algebra is usually denoted by $\big(\frac{a,b}{K}\big)$. 
\end{definition}

In particular, when $K=\RR$ and $a=b=-1$ (then $\ii=\sqrt{-1} = i$, $\jj=\sqrt{-1} = j$ and $\ii \jj=ij = k$), the quaternion algebra $\big(\frac{-1,-1}{\RR}\big)$ is the classical algebra of \emph{Hamilton's quaternions}, which is usually denoted by $\HH$. More generally, if $\B = \big(\frac{a,b}{\RR}\big)$ is a quaternion algebra over $\RR$, then $\B \simeq \mbox{M}_2(\RR)$ or $\B \simeq \HH$ (this last case occurs if and only if $a,b<0$). On the other hand, when $K= \CC$, we have that $\B = \big(\frac{a,b}{\CC}\big) \simeq \mbox{M}_2(\CC)$ for all $a,b \in \CC^\times$.

Inspired by the complex conjugation we can define a \emph{standard involution} on $\B$ given by the map 
\[
\overline{\phantom{q}} : \B \longrightarrow \B
\]
\[
\q = x_0+x_1 \ii +x_2 \jj +x_3 \ii \jj  \longmapsto \overline{\q} = x_0 - x_1 \ii - x_2 \jj - x_3 \ii \jj.
\]
The existence of such involution allows us to define a  \emph{reduced trace} $\trd: \B \rightarrow K$ by $\trd(\q) := \q+\overline {\q}$ and a \emph{reduced norm} $\nrd: \B \rightarrow K$ by $\nrd(\q) := \q \overline \q$. Then, we can define $\Re(\q) := \frac{\trd(\q)}{2}$ and $\vert \q \vert := \nrd(\q)$. 
Of particular interest to us will be the $K$-subspace of \emph{pure} elements of $\B$ defined as
\[
\B^0 := \{ \q \in \B :  \trd(\q) = 0 \}
\]
and the normal subgroup 
\[
\B^1 := \{ \q \in \B^\times : \nrd(\q) = 1 \}
\]
of $\B^\times$ of elements of reduced norm 1.

When $\B \simeq \HH$, we have that $\HH^0$ is the subspace of classical \emph{pure Hamiltonians} and  $\HH^1$ is the classical subgroup of \emph{unit Hamiltonians}. 
As a set, the unit Hamiltonians are naturally identified with the 3-sphere $\mathbb{S}^3$ in $\RR^4$. The group $\HH^1$ acts  by rotation on $\HH^0 \simeq \RR^3$ (on the left) via conjugation $\omega  \mapsto \q \omega \q^{-1}$. This action defines a group homomorphism $\HH^1 \rightarrow \SO(3)$, fitting into the following exact sequence
\[
1 \longrightarrow \{ \pm 1 \} \longrightarrow \HH^1 \longrightarrow \SO(3) \longrightarrow 1 .
\]
The interpretation of the quotient group $\HH^1 / \{ \pm 1 \}$ as the group of rotations of $\RR^3$ is very useful to determine certain algebraic substructures in $\HH^1$. For example, from the classification of finite groups of $\SO(3)$  \cite[Proposition 11.5.2]{Voight2021}, we have that each finite subgroup of $\HH^1$ is isomorphic to one of the following groups \cite[Ch. I, Théorème 3.7]{Vigneras1980}:
\begin{enumerate}
\item a cyclic group $C_n$ of order $n$ generated by $s_n = \cos (2\pi /n) + i \sin (2 \pi /n)$;
\item a \emph{binary dihedral (dicyclic)} group $Q_{4n}$ of order $4n$ generated by $s_{2n}$ and $j$;
\item the \emph{binary tetrahedral} group $2T$ of order 24 with presentation given by
\[
\langle r,s,t \mid r^{2}=s^{3}=t^{3}=rst=1 \rangle 
\]
where $r= i$, $s= \frac{1}{2} (1+i+j+k)$ and $t= \frac{1}{2}(1+i+j-k)$;
\item the \emph{binary octahedral} group $2O$ of order 48 with presentation given by
\[
\langle r,s,t \mid r^{2}=s^{3}=t^{4}=rst=1 \rangle 
\]
where $r= \frac{1}{\sqrt{2}}(i+j)$, $s= \frac{1}{2} (1+i+j+k)$ and $t= \frac{1}{\sqrt{2}}(1+i)$; or
\item the \emph{binary icosahedral} group $2I$ of order 120 with presentation given by
\[
\langle s,t \mid (st)^{2}=s^{3}=t^{5}=rst=1 \rangle 
\]
where $s= \frac{1}{2} (1+i+j+k)$, $t= \frac{1}{2}(\varphi+\varphi^{-1}i + j)$ and $\varphi = \frac{1+\sqrt{5}}{2}$ is the golden ratio.
\end{enumerate}


\subsection{Möbius transformations}\label{Sec:22} 

Let $\Mat_2(\HH)$ be the $\HH$-module (right or left, according to the setting) of $2 \times 2$ matrices with entries in $\HH$. It can be proven that all right-invertible matrices in $\Mat_2(\HH)$ are also left-invertible \cite[Proposition 2.3]{BisiGentili2008}. Then, we can define the \emph{general linear group} $\GL_2(\HH)$ as the set of all invertible matrices of $\Mat_2(\HH)$.

\begin{definition}
Let $\hat{\HH} := \HH \cup \{ \infty \}$ be the compactification of $\HH$ and  
 $\gamma = \left(
\begin{matrix}
a &  b\\
c & d
\end{matrix}
\right) \in \GL_2(\HH)$.
We define the \emph{quaternionic Möbius transformation} associated to $\gamma$ as the real analytic function $F_\gamma : \hat{\HH} \rightarrow  \hat{\HH}$ given by 
\begin{equation}\label{mobi}
F_\gamma (\q) := (a \q +b) \cdot (c \q + d)^{-1},
\end{equation}
where we set $F_\gamma(\infty) = \infty$ if $c = 0$, $F_\gamma(\infty) = ac^{-1}$ if $c \neq 0$, and  $F_\gamma (-c^{-1}d) = \infty$.  
\end{definition}

As in the complex case, every quaternionic Möbius transformation is a composition of homotheties, translations, and inversions. More precisely, let $b,c \in \HH$, with $c\neq0$. We define the \emph{left homothetic transformation} $h_c: \HH \rightarrow \HH$ as the map $\q \mapsto c\q$, the \emph{translation} $T_b : \HH \rightarrow \HH$ as the map $\q \mapsto \q+b$, and the \emph{inversion} $I : \HH \rightarrow \HH$ as the map $\q \mapsto \q^{-1} = \frac{\overline{\q}}{ \vert \q \vert ^2}$. Then, $F_\gamma(\mathbf{q})$ can be decomposed as follows:
\[
\mathbf{q}\xrightarrow{T_{c^{-1}d}}(\mathbf{q}+c^{-1}d)
\xrightarrow{h_c}{c\mathbf{q}+d}
\xrightarrow{I}(c\mathbf{q}+d)^{-1}\xrightarrow{h_{b-ac^{-1}d}}(b-ac^{-1}d)(c\mathbf{q}+d)^{-1} \xrightarrow{T_{ac^{-1}}} (b-ac^{-1}d)(c\mathbf{q}+d)^{-1}+ac^{-1}
\]
\[
=(b-ac^{-1}d)(c\mathbf{q}+d)^{-1}+ac^{-1}(c\mathbf{q}+d)(c\mathbf{q}+d)^{-1}=(a\mathbf{q}+b)(c\mathbf{q}+d)^{-1} .
\]
Therefore 
\begin{equation}\label{gene}
F_\gamma=T_{ac^{-1}}\circ{h_{b-ac^{-1}d}}\circ{I}\circ{h_c}\circ{T_{c^{-1}d}}.
\end{equation}

Let $\M := \{ F_\gamma : \gamma \in \GL_2(\HH) \}$ be the group of quaternionic Möbius transformations with respect to the composition operation. The map $\Phi: \GL_2(\HH) \rightarrow \M$, defined as $\Phi(\gamma) := F_{\gamma}$, is a surjective group antihomomorphism with $\ker(\Phi) := \{ t \mathcal{I} : t \in \RR^\times \}$, where $\mathcal{I} \in \GL_2(\HH)$ is the identity matrix. As the quaternionic Möbius transformations are orientation-preserving conformal diffeomorphisms of $\hat{\HH}$, which can be identified with the 4-sphere $\mathbb{S}^4$ via the stereographic projection, then
\[
\PSL(2,\HH) := \GL_2(\HH)/\{t \mathcal{I} :  t \in \RR^\times \} \simeq Conf_+(\mathbb{S}^4),
\] 
where $Conf_+(\mathbb{S}^4)$ denotes the group of orientation-preserving conformal diffeomorphisms of the 4-sphere. 
By abuse of language, we will sometimes identify a matrix with quaternionic coefficients with the induced Möbius transformation and vice versa.

Now, let 
\[
\Hi_{\HH}^1 := \{ \q \in \HH : \Re(\q) >0 \} \subseteq \HH
\]
be the half-space model of the \emph{one-dimensional quaternionic hyperbolic space}.
This space is isometric to the real hyperbolic 4-space $\Hi_{\RR}^4 := \{ (x_0 , x_1 , x_2 , x_3) \in \RR^4 : x_0 >0 \}$ with the Poincare metric $(ds)^2 = \frac{(dx_0)^2+(dx_1)^2+(dx_2)^2+(dx_3)^2}{x_0^2}$. 

Let $\M_{\Hi_{\HH}^1} \subseteq \PSL_2(\HH)$ be the subgroup of quaternionic Möbius transformation leave invariant $\Hi_{\HH}^1$.
As any $F_\gamma \in \M_{\Hi_{\HH}^1}$ is conformal and preserves orientation, moreover is an isometry of $\Hi_{\HH}^1$, then $\M_{\Hi_{\HH}^1}$ is isomorphic to the groups $Conf_+(\Hi_{\HH}^1)$ and $Isom_+(\Hi_{\HH}^1)$ of conformal diffeomorphisms and isometries orientation-preserving of the half-space model $\Hi_{\HH}^1$. 
In addition, $\mathcal{M}_{\Hi_{\HH}^1}$ acts by orientation-preserving conformal transformations on the sphere at infinity of the hyperbolic 4-space defined as $\partial \Hi_{\HH}^1 := \{ \q \in \HH : \Re(\q) =0 \} \cup \{ \infty \}$. Then, $\M_{\Hi_{\HH}^1} \cong Conf_+(\mathbb{S}^3)$. 

We remark that $\M_{\Hi_{\HH}^1} \subseteq \PSL_2(\HH)$ can be characterized as the group induced by matrices which satisfy one of the following equivalent (Bisi-Gentili)\emph{BG-conditions} \cite[Theorem 1.11]{BisiGentili2008}  \cite[Proposition 2.3]{DiazVerjovskyVlacci2015} (which are a variation of the conditions described by Ahlfors \cite{Ahlfors1985} and  Vahlen \cite{Vahlen1902}):
\begin{equation}
\left\lbrace \gamma \in \PSL_2(\HH) 
: \overline{\gamma}^\top \left(
\begin{matrix}
0 &  1\\
1 & 0
\end{matrix}
\right) \gamma = \left(
\begin{matrix}
0 &  1\\
1 & 0
\end{matrix}
\right)  \right\rbrace ,
\end{equation}
\begin{equation}
\left\lbrace \left(
\begin{matrix}
a &  b\\
c & d
\end{matrix}
\right) \in \PSL_2(\HH) 
: \Re(a \overline{c}) =0, \; \Re(b \overline{d}) =0, \;  \overline{b}c + \overline{d}a =1  \right\rbrace ,
\end{equation}
\begin{equation}
 \left\lbrace \left(
\begin{matrix}
a &  b\\
c & d
\end{matrix}
\right) \in \PSL_2( \HH) 
: \Re(c \overline{d}) =0, \; \Re(a \overline{b}) =0, \;  a\overline{d} + b\overline{c} =1  \right\rbrace.
\end{equation}

An important subgroup of $\mathcal{M}_{\Hi_{\HH}^1}$ is the \emph{affine subgroup} $\A(\HH)$ consisting of transformations which are induced by matrices of the form $ \left(
\begin{matrix}
\lambda a &  b\\
0 & \lambda^{-1} a
\end{matrix}
\right)$
with $\vert a \vert = 1$, $\lambda > 0$ and $\Re(a\overline{b}) = 0$ (which clearly satisfy BG-conditions). The group $\A(\HH)$ is the maximal subgroup of $\M_{\Hi_{\HH}^1}$, which fixes the point at infinity, and any $F_\gamma \in \A(\HH)$ acts as a conformal transformation on $\partial \Hi_{\HH}^1$. Moreover, $\A(\HH)$ is the group of conformal and orientation-preserving transformations acting on the space of pure quaternions at infinity, which can be identified with $\RR^3$ so that $\A(\HH) \cong Conf_+(\RR^3)$.

A useful decomposition of the elements of $\M_{\Hi_{\HH}^1}$ is the \emph{Iwasawa decomposition} \cite[Proposition 2.4]{DiazVerjovskyVlacci2015}, which states that every $\gamma \in \mathcal{M}_{\Hi_{\HH}^1}$ can be written in the form
\begin{align}\label{iwasawa}
\gamma=\begin{pmatrix}
		\lambda & 0\\
		0 & \lambda^{-1}
	\end{pmatrix}
	\begin{pmatrix}
		1 & \omega \\
		0 & 1
	\end{pmatrix}
	\begin{pmatrix}
		\alpha & \beta\\
		\beta & \alpha
	\end{pmatrix}
\end{align}
where $\lambda\in\RR^+$, $\omega \in \HH^0$, and $\alpha,\beta\in\HH$ satisfy $\vert \alpha \vert^2 + \vert \beta \vert^2=1$ and $\Re(\alpha \overline{\beta}) = 0$. The first matrix is a homothety fixing $0$ and $\infty$, the second matrix is a parabolic translation fixing $\infty$ in the direction of $\omega$ and the third matrix is a 4-dimensional rotation. In fact, the set $\mathcal{K}$, of all matrices of the form 
$\begin{pmatrix}
		\alpha & \beta\\
		\beta & \alpha
	\end{pmatrix}$,
is isomorphic to the special orthogonal group $\SO(4)$ which is a real compact Lie group of dimension 6. Then, we can deduce that the set of all matrices of $\PSL_2(\HH)$ satisfying the BG-condition has real dimension 10. 
In particular, it is important to note that the action at infinity of the subset of diagonal matrices $\mathcal{D} :=  \left\lbrace \begin{pmatrix}
		\alpha & 0 \\
		0 & \alpha
	\end{pmatrix} \in  \mathcal{K} \right\rbrace$,
 is given by $\q \mapsto \alpha \q \overline{\alpha}$, which is the usual action of $\SO(3)$ in the space of pure Hamiltinians $\HH^0$. Then, $\mathcal{D}$ is isomorphic to $\SO(3)$.	


\section{Orders in totally definite quaternion algebras}\label{sec:3}

Let $K$ be a number field and $\ZZ_K$ denotes its ring of integers. For example, if $K=\QQ$ then $\ZZ_K$ is simply $\ZZ$ and if $K$ is a quadratic field (i.e., $K=\QQ(\sqrt{n})$ for some nontrivial square-free $n \in \ZZ$ different from 0 and 1) then $\ZZ_K = \ZZ[\theta]:= \ZZ \oplus \ZZ \theta$ where
 \[
\theta := \begin{cases} \sqrt{n} & \mbox{if } n \nequiv 1 \mod 4 \\ \frac{1+\sqrt{n}}{2} & \mbox{if } n \equiv 1 \mod 4 \end{cases} .
\]

\begin{definition}
A $\ZZ_K$-\emph{order} $\Ord$, in a quaternion algebra $\B$ over $K$, is a $4$-dimensional $\ZZ_K$-lattice in $\B$ that is also a ring with unity. Moreover, $\Ord$ is called \emph{maximal} if no other $\ZZ_K$-order properly contains it.
\end{definition}

\begin{remark}
Maximal $\ZZ_K$-orders are analogous to rings of integers of number fields but an important difference is that rings of integers are unique, while a quaternion algebra can have many maximal $\ZZ_K$-orders. For example, if $\Ord \subseteq \B$ is a maximal $\ZZ_K$-order and $\q \in \B^\times$, then $\q \Ord \q^{-1} \subseteq \B$ is a maximal $\ZZ_K$-order,  but as $\B$ is noncommutative, we may have $\q \Ord \q^{-1} \neq \Ord$.
\end{remark}

Let $a,b \in \ZZ_K \backslash \{ 0 \}$ and $\B = \big(\frac{a,b}{K}\big)$. The most natural example of a $\ZZ_K$-order is the \emph{standard order} in $\B$ 
\[
\Ord_\B := \ZZ_K \oplus \ZZ_K \ii \oplus \ZZ_K \jj \oplus \ZZ_K \ii \jj.
\] 
In particular, when $ \B = \B_K$, $\HH(\ZZ_K) := \Ord_{\B_K} = \ZZ_K \oplus \ZZ_K i \oplus \ZZ_K j \oplus \ZZ_K k$ is properly contained in the $\ZZ_K$-order 
\[
\Hur(\ZZ_K) := \ZZ_K \oplus \ZZ_K i \oplus \ZZ_K j \oplus \ZZ_K \xi,
\] 
where $\xi = \frac{1+i+j+k}{2}$, showing that $\HH(\ZZ_K)$ is never a maximal $\ZZ_K$-order in $\B_K$. These orders generalize the rings of Lipschiz and Hurwitz integers $\HH(\ZZ)$ and $\HH ur(\ZZ)$ studied in \cite{DiazVerjovskyVlacci2015}. However, in contrast to the maximality of $\HH ur(\ZZ)$ in $\B_{\QQ} = \big(\frac{-1,-1}{\QQ}\big)$, $\Hur(\ZZ_K)$ is not always maximal in $\B_K$. For example, if $K = \QQ(\sqrt{2})$ we can define the \emph{binary octahedral order} of $\B_{\QQ(\sqrt 2)}$ as
\[
\OrdO := \ZZ [\sqrt2] \oplus  \ZZ [\sqrt2] \eta \oplus \ZZ [\sqrt2] \delta \oplus \ZZ [\sqrt2] \eta\delta,
\]
where $\eta=\frac{1+i}{\sqrt 2}$ and $\delta =\frac{1+j}{\sqrt 2}$, which properly contains $\Hur(\ZZ[\sqrt{2}])$. 
Similarly, when $K= \QQ(\sqrt5)$, we can define the \emph{binary icosahedral order} of $\B_{\QQ(\sqrt 5)}$ as
\[
\OrdI := \ZZ [\varphi] \oplus  \ZZ [\varphi] i \oplus \ZZ [\varphi] \zeta \oplus \ZZ [\varphi] i \zeta,
\]
where $\varphi = \frac{1+\sqrt5}{2}$ is the golden ratio and $\zeta = \frac{\varphi + \varphi^{-1} i + j}{2}$, which properly contains $\Hur(\ZZ[\varphi])$. In fact, $\OrdO$ and $\OrdI$ are maximal $\ZZ_K$-orders of $\B_{\QQ(\sqrt2)}$ and $\B_{\QQ(\sqrt5)}$ respectively (see \cite[\S 11.5]{Voight2021}).


\subsection{Totally definite quaternion algebras}\label{sec:31}
Let $\sigma: K \hookrightarrow \CC$ be an embedding that induces the archimedean absolute value $\vert x \vert _\sigma  := \vert \sigma(x) \vert$. It is well-known that, up to equivalence, every archimedean absolute value of $K$ is induced by an embedding of $K $ into $\CC$ and two such embeddings give rise to equivalent absolute values if and only if they coincide or differ by complex conjugation. By an \emph{archimedean place} $\sigma$ of $K$ we mean an equivalence class of an archimedean absolute value $\vert \cdot \vert_\sigma$ on $K$ and we denote by $K_\sigma$ the completion of $K$ with respect to $\vert \cdot \vert_\sigma$ (also called the completion of $K$ at the place $\sigma$) which is isomorphic to $\RR$ or $\CC$. 
In particular, $K$ is called \emph{totally real} if $K_\sigma \cong \RR$ for each archimedean places $\sigma$ of $K$.

\begin{definition}
A quaternion algebra $\B$ over a number field $K$ is \emph{totally definite} if for all archimedian places $\sigma$ of $K$, $\B_\sigma := \B \otimes_{\QQ} K_\sigma \simeq \HH$. 
\end{definition}

We remark that, when $\sigma$ is a complex place, we have that $\B_\sigma = \B \otimes_{\QQ} \CC \simeq \mbox{M}_2(\CC)$. Therefore, if $\B$ is a totally definite quaternion algebra over $K$, then $K$ is necessarily a totally real field. In particular, we will denote by $\B_K$ the totally definite quaternion algebra $\big(\frac{-1,-1}{K}\big)$ defined over the totally real field $K$.  
Henceforth, we will assume that $\B_K$ is totally definite. Then, it can be embedded in $\HH$ and, by restriction, we can embed any $\ZZ_K$-order $\Ord$ of $\B_K$ in $\HH$.   

Now, let $\gamma = \left(
\begin{matrix}
a &  b\\
c & d
\end{matrix}
\right) \in \Mat_2(\HH)$.
Recall that the \emph{Dieudonné determinant} of $\gamma $ is defined as the nonnegative real number
\[
{\det}_\HH(\gamma) := \sqrt{\vert a \vert ^2 \vert d \vert ^2 + \vert c \vert ^2 \vert b \vert ^2 - 2 \Re (c \overline{a} b \overline{d})}.
\]
It can be proven that a matrix $\gamma \in \Mat_2(\HH)$ is invertible if and only if ${\det}_\HH(\gamma) \neq 0$ \cite[Proposition 2.8]{BisiGentili2008}. Then, $\GL_2(\HH)$ is precisely the set of all matrices in $\Mat_2(\HH)$ having non zero Dieudonné determinant. Thus, we can define the \emph{special linear group}  $\SL_2(\HH)$ as the set of all matrices in $\GL_2(\HH)$ with Dieudonné determinant 1 and $\PSL_2(\HH) \cong \SL_2(\HH) / \{ \pm \mathcal{I} \}$.

Of particular interest for us will be the subset $\SL_2(\Ord) \subseteq \SL_2(\HH)$ of all matrices with coefficients in a $\ZZ_K$-order $\Ord$ of a totally definite quaternion algebra $\B$, which in fact is a subgroup. This is perhaps well-known to the experts, but we give a proof for lack of a reference.

\begin{lemma}
Let $\B$ be a totally definite quaternion algebra over $K$ and $\Ord$ be a $\ZZ_K$-order in $\B$. Then, $\SL_2(\Ord)$ is a subgroup of $\SL_2(\HH)$.
\end{lemma}

\begin{proof}
It is clear that the right product of matrices in $\SL_2(\Ord)$ is well defined, associative and $\mathcal{I} \in \SL_2(\Ord)$ because $\Ord$ is a ring with unity. Then, we only need to prove that the right-inverse of each $\gamma = \begin{pmatrix}
a & b \\
c & d
\end{pmatrix} \in \SL_2(\Ord)$ lies in $\SL_2(\Ord)$. 

First assume that $abcd \neq 0$, then by \cite[Remark 2.2]{BisiGentili2008}, we have that
\[
\gamma ^{-1} = \begin{pmatrix}
(a-bd^{-1}c )^{-1} & (c-db^{-1}a )^{-1} \\
(b-ac^{-1}d )^{-1} & (d-ca^{-1}b )^{-1}
\end{pmatrix}
\]
with $a^{-1}$, $b^{-1}$, $c^{-1}$, $d^{-1}$ not necessarily in $\Ord$. Using the multiplicity of the norm and \cite[Lemma 2.4]{BisiGentili2008}, we have that 
\[
\vert a \vert^2 \vert d-ca^{-1} b \vert^2 = \vert b \vert^2 \vert c-db^{-1} a \vert^2 = \vert c \vert^2 \vert b-ac^{-1} d \vert^2 = \vert d \vert^2 \vert a-bd^{-1} c \vert^2
= \vert a \vert^2 \vert d \vert^2 + \vert c \vert^2 \vert b \vert^2 -  2 \Re (c \overline{a}b\overline{d} ) = 1.
\]
Then, as $\mathbf{q}^{-1} = \frac{\overline{\mathbf{q}}}{\vert \mathbf{q} \vert ^2}$ and $\mathbf{\overline{q}}^{-1} = \frac{\mathbf{q}}{\vert \mathbf{q} \vert ^2}$,
\[
\gamma ^{-1} = \begin{pmatrix}
(a-bd^{-1}c )^{-1} & (c-db^{-1}a )^{-1} \\
(b-ac^{-1}d )^{-1} & (d-ca^{-1}b )^{-1}
\end{pmatrix} =
\begin{pmatrix}
\frac{(\overline{a} - \overline{b} \overline{d}^{-1} \overline{c})}{ \vert a-bd^{-1}c  \vert^2} & \frac{(\overline{c} - \overline{d} \overline{b}^{-1} \overline{a})}{ \vert c-db^{-1}a  \vert^2} \\
\frac{(\overline{b} - \overline{a} \overline{c}^{-1} \overline{d})}{ \vert b-ac^{-1}d  \vert^2}   & \frac{(\overline{d} - \overline{c} \overline{a}^{-1} \overline{b})}{ \vert d-ca^{-1}b  \vert^2}
\end{pmatrix}
= 
\begin{pmatrix}
\frac{\vert d \vert ^2}{ \vert d \vert ^2 } \frac{(\overline{a} - \overline{b} \overline{d}^{-1} \overline{c})}{ \vert a-bd^{-1}c  \vert^2} & \frac{\vert b \vert ^2}{ \vert b \vert ^2 } \frac{(\overline{c} - \overline{d} \overline{b}^{-1} \overline{a})}{ \vert c-db^{-1}a  \vert^2} \\
\frac{\vert c \vert ^2}{ \vert c \vert ^2 } \frac{(\overline{b} - \overline{a} \overline{c}^{-1} \overline{d})}{ \vert b-ac^{-1}d  \vert^2}   & \frac{\vert a \vert ^2}{ \vert a \vert ^2 } \frac{(\overline{d} - \overline{c} \overline{a}^{-1} \overline{b})}{ \vert d-ca^{-1}b  \vert^2}
\end{pmatrix}
\]
\[= 
\begin{pmatrix}
\vert d \vert ^2 (\overline{a} - \overline{b} \overline{d}^{-1} \overline{c}) & \vert b \vert ^2 (\overline{c} - \overline{d} \overline{b}^{-1} \overline{a}) \\
\vert c \vert ^2 (\overline{b} - \overline{a} \overline{c}^{-1} \overline{d})   & \vert a \vert ^2 (\overline{d} - \overline{c} \overline{a}^{-1} \overline{b})
\end{pmatrix}
= \begin{pmatrix}
\vert d \vert ^2 (\overline{a} - \overline{b} \frac{d}{ \vert d \vert ^2} \overline{c}) & \vert b \vert ^2 (\overline{c} - \overline{d} \frac{b}{ \vert b \vert ^2} \overline{a}) \\
\vert c \vert ^2 (\overline{b} - \overline{a} \frac{c}{ \vert c \vert ^2} \overline{d})   & \vert a \vert ^2 (\overline{d} - \overline{c} \frac{a}{ \vert a \vert ^2} \overline{b})
\end{pmatrix}
= \begin{pmatrix}
\vert d \vert ^2 \overline{a} - \overline{b} d \overline{c} & \vert b \vert ^2  \overline{c} - \overline{d} b \overline{a} \\
\vert c \vert ^2 \overline{b} - \overline{a} c \overline{d}  & \vert a \vert ^2 \overline{d} - \overline{c} a \overline{b}
\end{pmatrix}
\]
with all its coefficients in $\Ord$. 

Now assume that one of the entries of $\gamma$ is $0$. For example, if $a=0$ and $bc \neq 0$, by (2.3) of \cite{BisiGentili2008}, we have that
\[
\gamma ^{-1} = \begin{pmatrix}
-c^{-1} d b^{-1} & c^{-1} \\
b^{-1} & 0
\end{pmatrix}
\]
with $b^{-1}$ and $c^{-1}$ not necessarily in $\Ord$. Then, as the Dieudonné determinant is equal to 1, we have that $\vert c \vert ^2 \vert b \vert ^2 = 1$ and 
\[
\gamma ^{-1} = \begin{pmatrix}
-c^{-1} d b^{-1} & c^{-1} \\
b^{-1} & 0
\end{pmatrix} 
=
\begin{pmatrix}
-\frac{\overline{c}}{\vert c \vert ^2} d \frac{\overline{b}}{\vert b \vert ^2} & \frac{\overline{c}}{\vert c \vert ^2} \\
\frac{\overline{b}}{\vert b \vert ^2} & 0
\end{pmatrix}
=
\begin{pmatrix}
-\frac{1}{\vert c \vert ^2 \vert b \vert ^2} \overline{c} d \overline{b} & \frac{\overline{c}}{\vert c \vert ^2} \frac{\vert b \vert ^2}{\vert b \vert ^2} \\
\frac{\vert c \vert ^2}{\vert c \vert ^2} \frac{\overline{b}}{\vert b \vert ^2} & 0
\end{pmatrix}
=
\begin{pmatrix}
- \overline{c} d \overline{b} & \vert b \vert ^2 \overline{c} \\
\vert c \vert ^2 \overline{b} & 0
\end{pmatrix}
\]
with all its coefficients in $\Ord$. The other cases are analogous.
\end{proof}


\subsection{Unit groups of quaternion orders}\label{sec:32}

In number theory, a very important subgroup of $\ZZ_K$ is the unit group $\ZZ_K^{\times}$ whose structure is well-understood thanks to Dirichlet’s unit theorem \cite[Chapter I \S 7]{Neukirch1999}. On the other hand, a description of the \emph{unit group} $\Ord^{\times}$ of a $\ZZ_K$-order $\Ord$ in an arbitrary quaternion algebra is more complicated. 
However, as we are only interested in $\ZZ_K$-orders in a totally definite quaternion algebras $\B$ over $K$, we can describe $\Ord^{\times}$ in terms of $\ZZ_K^\times$ and the \emph{torsion group} $\Ord^1 := \{ \uu \in \Ord^\times :   \nrd (\uu) = 1 \}$ (which is a finite subgroup of $\Ord^\times$ \cite[Proposition 32.3.7]{Voight2021}) of $\Ord$ as follows. 

Let $J_K$ be the set of all embeddings of $K$ into $\CC$ (in fact, into $\RR$ because $K$ is totally real). It is well known that there are precisely $[K: \QQ]$ distinct such embeddings.
Let $K^\times_+ := \{ w \in  K^\times : \sigma(w)>0 \mbox{ for all } \sigma \in J_K \}$ be the set of totally positive elements of $K^\times$ and $\ZZ_{K + }^\times := \ZZ_{K}^\times \cap K^\times_+$ be the group of totally positive units in $\ZZ_K$. As $\nrd(\Ord)  \subseteq  \ZZ_{K}$ \cite[\S 10.3]{Voight2021} and $\nrd(\B^\times) \subseteq K^\times_+$  \cite[\S 14.7]{Voight2021}, we have that $\nrd(\Ord^\times)  \subseteq  \ZZ_{K + }^\times$. Then, the reduced norm induces the following exact sequence 
\[
1 \longrightarrow \Ord^1 \longrightarrow \Ord^\times \xrightarrow{\nrd} \ZZ_{K + }^\times 
\] 
which implies that $\Ord^1$ is a normal subgroup of $\Ord^\times$. On the other hand, as $Z(\B^\times) = K^\times$, we have that $\ZZ_K^\times \subseteq Z (\Ord ^\times)$, which implies that $\ZZ_K^\times$ is a normal subgroup of  $\Ord^\times$. Thus, $\ZZ_K^\times \Ord^1$ is a normal subgroup of $\Ord^\times$ and, since $\nrd(\ZZ^\times_K \Ord^1) = \ZZ_K^{\times 2}$, we have the following embedding  
\[
\Ord^\times / \ZZ_K^\times \Ord^1 \hookrightarrow \ZZ_{K + }^\times / \ZZ_K^{\times 2}.
\]
Finally, as $\ZZ^\times_K$ is finitely generated \cite[\S 13.3]{AlacaWilliams2004},  $\ZZ_{K + }^\times / \ZZ_K^{\times 2}$ is a finite elementary abelian 2-group. Therefore, $\Ord^\times$ is an extension of $\ZZ_K^\times \Ord^1$ by an elementary abelian 2-group. In fact, Vignéras proved in \cite[Théorème 6]{Vigneras1976} that $[\Ord^\times: \ZZ_K^\times \Ord^1]$ is equal to 1, 2 or 4.

When $K=\QQ$,  it can be proved that $\Ord^\times = \Ord^1 \subseteq \HH^1$ \cite[Lemma 11.5.9]{Voight2021}. Then, $\Ord^\times$ is isomorphic to one of the finite subgroups of $\HH^1$ described above. Such description was crucial in \cite{DiazVerjovskyVlacci2015} to describe the Lipschitz and Hurwitz quaternionic modular groups as well as their fundamental domains. In fact, it can be proved \cite[Chapitre V, Proposition 3.1]{Vigneras1980} that, given a (totally) definite quaternion algebra $\B = \big( \frac{a,b}{\QQ}\big)$ the unit group $\Ord^\times$ of a maximal $\ZZ$-order $\Ord \subseteq \B$ is cyclic of order 2, 4 or 6, except for 
\begin{itemize}
\item $\B = \big( \frac{-1,-3}{\QQ}\big)$ where $\Ord^\times$ is isomorphic to the binary dihedral group $Q_{12} = \langle s_6, j \rangle$; and
\item $\B = \big( \frac{-1,-1}{\QQ}\big)$, where $\Ord^\times \simeq \HH ur(\ZZ)^1$, which is isomorphic to the binary tetrahedral group $2T$.
\end{itemize}

When $K = \QQ(\sqrt{n})$ is a totally real field, a description of $\Ord^\times$ is a bit more complicated because it is no longer finite, but it can be done.  In this case,  by using Dirichlet's unit theorem \cite[Theorem 11.5.1]{AlacaWilliams2004}, we have that
 \begin{equation}\label{DUT}
\ZZ_K^{\times} =  \{ \pm \varepsilon ^{\ell} : \ell \in \ZZ \},
\end{equation}
where $\varepsilon$ is the fundamental unit of $\ZZ_K$ normalized so that $\varepsilon > 1$ for the canonical embedding $K \hookrightarrow \RR$.
Since $K$ is a real quadratic field, it is well known \cite[Chapter 11]{AlacaWilliams2004} that $\ZZ_K^{\times 2} = \langle \varepsilon ^2 \rangle$ and
\[
\ZZ_{K +}^\times = \begin{cases} \langle \varepsilon \rangle & \mbox{if } N_{K/ \QQ} (\varepsilon) = 1 \\ \langle \varepsilon ^2 \rangle & \mbox{if } N_{K/\QQ}(\varepsilon) = -1. \end{cases}
\]
Then, if $N_{K/ \QQ} (\varepsilon) = -1$, we conclude that $\Ord^\times \simeq \ZZ_K^\times \Ord^1 $ and if $N_{K/ \QQ} (\varepsilon) = 1$, we conclude that $\Ord^\times$ is isomorphic to $\ZZ_K^\times \Ord^1$ or to a degree two extension of $\ZZ_K^\times \Ord^1$. 
Finally, as $K$ is totally real, the finite subgroup $\Ord^1$ is embedded in $\HH^1$. Then, we can obtain an explicit description of $\Ord^\times$ by using  (\ref{DUT}) and the description of the finite subgroups of $\HH^1$ given in the previous section.  

For example, keeping with the maximal $\ZZ_K$-orders $\OrdO \subseteq \B_{\QQ(\sqrt{2})}$ and $\OrdI \subseteq \B_{\QQ(\sqrt{5})}$ mentioned above, let $\varepsilon = 1+\sqrt{2}$ and $ \varepsilon = \varphi = \frac{1+\sqrt{5}}{2}$ be the normalized fundamental units of $\ZZ[\sqrt{2}]$ and $\ZZ[\varphi]$ respectively. From \cite[\S 11.5]{Voight2021}, we have that $\OrdO^1$ is isomorphic to the binary octahedral group $2O$ and $\OrdI^1$ is isomorphic icosahedral group $2I$. Then, as $N_{\QQ(\sqrt{2})/\QQ}(1+\sqrt{2})=-1$ and $N_{\QQ(\sqrt{5})/\QQ}(\varphi)=-1$, we have that 
\[
\OrdO^\times \simeq 2O \cdot \ZZ^\times_{\QQ(\sqrt{2})} \quad \mbox{ and } \quad \OrdI^\times \simeq 2I \cdot \ZZ^\times_{\QQ(\sqrt{5})}. 
\]

An example where $\Hur(\ZZ_K)$ is a maximal $\ZZ_K$-order is the following. Let $K=\QQ(\sqrt{13})$ and $\varepsilon = \frac{3+\sqrt{13}}{2}$ be the normalized fundamental unit of $\ZZ_{\QQ(\sqrt{13})}= \ZZ[\frac{1+\sqrt{13}}{2}]$. As $2\varepsilon = 2(\frac{9+\sqrt{13}}{2}) \notin \QQ(\sqrt{13} )^{\times 2}$, we have from \cite[Table 4.3]{Li-Xue-Yu2021} that there are not $\ZZ_{\QQ(\sqrt{13})}$-orders $\Ord$ in $\B_{\QQ(\sqrt{13})}$, with torsion group $\Ord^1$ isomorphic to $2O$ or $2I$. Then, it follows from the classification of finite subgroups of $\HH^1$ that $\Hur(\ZZ_{\QQ(\sqrt{13})})$ is maximal. On the other hand, as $N_{\QQ(\sqrt{13})/\QQ}(\frac{3+\sqrt{13}}{2})=-1$, we obtain 
\[
\Hur(\ZZ_{\QQ(\sqrt{13})})^\times \simeq \Hur(\ZZ_{\QQ(\sqrt{13})})^1 \cdot \ZZ^\times_{\QQ(\sqrt{13})} \simeq 2T \cdot \ZZ^\times_{\QQ(\sqrt{13})}. 
\]

In fact, it follows from \cite[\S 32.7]{Voight2021} that the groups $2O$ and $2I$ only occur as torsion groups of a quaternion $\ZZ_K$-order $\Ord$ of $\B_K$ when the real quadratic field $K$ is $\QQ(\sqrt{2})$ and $\QQ(\sqrt{5})$ respectively. Moreover, it also follows from \emph{loc. cit.} that a totally definite quaternion algebra $\B$ over $K$ contains an order $\Ord$ with torsion group isomorphic to $2T$, $2O$ and $2I$ if and only if $\B \simeq \B_K$. This justifies our particular interest in $\B_K$ over the rest of totally definite quaternion algebras $\B$ in the next sections.


\section{Quaternionic modular groups}\label{sec:4}

In \cite{DiazVerjovskyVlacci2015} the Lipschitz and Hurwitz quaternionic modular groups, acting on $\Hi_{\HH}^1$, are introduced as a quaternionic analog of the classical modular group. Following with such analogy, in this section, we describe a quaternionic analog of the Bianchi and Hilbert-Blumenthal modular groups, which act on $\Hi^5_{\RR}$ and $\Hi_{\HH}^1 \times \Hi_{\HH}^1$ respectively.


\subsection{Bianchi quaternionic modular group}\label{sec:41}

This subsection closely follows Section 13 of the extended version of \cite{DiazVerjovskyVlacci2015} available on  arXiv:1503.07214. Let $\BB^5 \subseteq \RR^5$ be the closed 5-ball and identify the interior of $\BB^5$ with the real hyperbolic 5-space $\Hi^5_{\RR}$ which can be regarded as $\{ (\q,t) \mid \q \in \HH, t>0 \}$. 
As we have seen in $\S$ \ref{sec:2}, $\PSL_2(\HH)$ acts conformally on $\hat \HH \simeq \mathbb{S}^4 = \partial \BB^5$ by quaternionic Möbius transformations (\ref{mobi}).  Then, by Poincaré Extension Theorem, each $\gamma \in \PSL_2(\HH)$ extends canonically to a conformal diffeomorphism of $\BB^5$, which restricted to $\Hi^5_{\RR}$, is an orientation preserving isometry $\tilde{\gamma}$ of the open 5-disk $\DD^5$ with the Poincaré hyperbolic metric.  Reciprocally, any orientation preserving isometry of $\Hi^5_{\RR}$ extends canonically to the ideal boundary $\RR^4 \cup \{ \infty \} \cong \hat{\HH}$ as an element of $\PSL_2(\HH)$. Thus, the map $\gamma \mapsto \tilde{\gamma}$ is an isomorphism and $\PSL_2(\HH) = Isom_+(\BB^5) = Isom_+(\Hi^5_{\RR})$.
Explicitly, for each $\gamma = \begin{pmatrix}
a & b \\
c & d 
\end{pmatrix} \in \PSL_2(\HH)$, the {\it Poincaré's extension} $F_{\tilde{\gamma}} : \Hi^5_{\RR} \rightarrow \Hi^5_{\RR}$ of the quaternionic Möbius transformation $F_\gamma$ is given by (see Theorem 13.2 of loc. cit.):
\[
F_{\tilde{\gamma}}(\q,t) = \left(  \frac{1}{\vert c \q + d \vert^2 + \vert c \vert ^2 t^2} ((a \q + b) (\overline{\q c} + \overline{d}) + a \overline{c} t^2), \frac{t}{\vert c \q +d \vert ^2 + \vert c \vert ^2 t^2}    \right)
\]
and when $t =0$ it corresponds to the action of $\gamma$ in $\hat{\HH} = \HH \cup \{ \infty \}$ as described in Definition \ref{mobi}.

Let $\Gamma \subseteq \PSL_2(\HH)$ be a discrete subgroup acting isometrically on $\Hi^5_{\RR}$. Then,  it follows from standard facts about Kleinian groups that $\Gamma$ acts properly and discontinuously on $\Hi^5_{\RR}$. Thus, $M_\Gamma := \Hi^5_{\RR} / \Gamma$ is a complete real 5-dimensional hyperbolic orbifold. Examples of discrete subgroups of $\PSL_2(\HH)$ can be obtained from quaternion algebras as follows.

\begin{definition}
Let $\Ord$ be a $\ZZ$-order of the quaternion algebra $\B_{\QQ} \subseteq \HH$. We define the \emph{Bianchi quaternionic modular group} associated to $\Ord$ as the group $\PSL_2(\Ord) := \SL_2(\Ord)/\{ \pm \mathcal{I} \}$. 
\end{definition}

Clearly, the Bianchi quaternionic modular group $\PSL_2(\Ord) \subseteq \PSL_2(\HH)$ associated to a $\ZZ$-order $\Ord$ of $\B_{\QQ}$ is a discrete subgroup which acts isometrically on $\Hi^5_{\RR}$. Then, we define the \emph{Bianchi quaternionic modular orbifold} associated to $\Ord$ as
\[
M_\Ord := \Hi^5_{\RR} / \PSL_2(\Ord).
\]

Let $\mathfrak{B}_\Ord$ be a basis for (the lattice) $\Ord$. For example, $\mathfrak{B}_{\HH(\ZZ)} = \{ 1, i, j, k\}$ and $\mathfrak{B}_{\HH ur (\ZZ)} = \{1,i,j, \xi \}$ where $\xi = \frac{1+i+j+k}{2}$. As we mentioned in \S \ref{sec:32}, any $\ZZ$-order $\Ord$ in the quaternion algebra $\B_\QQ$ satisfies $\Ord^\times = \Ord^1$. Then, $\Ord$ is isomorphic to a finite subgroup of $\HH^1$, which can be described by generators and relations as in \S \ref{sec:21}. We denote by $\mathfrak{U}_\Ord$ the set of generators of  $\Ord^\times$. For example, $\mathfrak{U}_{\HH(\ZZ)} = \{ i, j \}$ and $\mathfrak{U}_{\HH ur (\ZZ)} = \{i,\xi, \tau \}$, where $\xi$ is as above and $\tau = \frac{1+i+j-k}{2}$. Thus, from (\ref{gene}) we can obtain the following (non-minimal) set of generators of $\PSL_2(\Ord)$:
\begin{equation}\label{gener}
\left\lbrace  
\begin{pmatrix}
0 & 1 \\
1 & 0
\end{pmatrix} ,
\begin{pmatrix}
1 & \omega \\
0 & 1
\end{pmatrix} ,
\begin{pmatrix}
\uu & 0 \\
0 & \uu
\end{pmatrix} : \omega \in \mathfrak{B}_\Ord, \; \uu \in \mathfrak{U}_\Ord
\right\rbrace
\end{equation}

Using this set of generators, it is easy to see that we can choose a fundamental domain $\mathcal{P}_\Ord$, for the action of the Poincaré extension of 
$\PSL_2(\Ord)$ on  $\Hi^5_{\RR}$, which is a subset of the $5$-dimensional chimney:
\[
\mathcal{P} := \{ (\q,t) = (x_0 + x_1 i + x_2 j + x_3 k, t) \in \Hi^5_{\RR} : - 1/2  \leq x_n \leq 1/2, \; (n=0,1,2,3) \mbox{ and } \vert \q \vert ^2 + \vert t \vert ^2 \geq 1 \}.
\]
Then, we have the following result, which provides us some explicit examples of complete real $5$-dimensional hyperbolic orbifold of finite hyperbolic volume.

\begin{proposition}
Let $\Ord$ be a $\ZZ$-order of the quaternion algebra $\B_{\QQ} \subseteq \HH$. Then, the Bianchi quaternionic modular orbifold $M_\Ord$ has finite hyperbolic volume.
\end{proposition}

In general, there is no reason to expect that an orbifold possesses a finite manifold cover. However, in our case, thanks to Selberg’s lemma \cite{Alperin1987} \cite{Selberg1960},  the groups $\PSL_2(\Ord)$  that define the orbifolds $M_\Ord$
 have torsion-free subgroups of finite index.  Therefore, for each $\ZZ$-order  $\Ord$,
 there exist a  finite orbifold cover $p: \widetilde{M}_\Ord \rightarrow M_\Ord$, where $\widetilde{M}_\Ord$ is an hyperbolic 5-manifold of finite volume.

Finally, we remark that if we define a cusp of $\PSL_2(\Ord)$ (equivalently of the orbifold $M_\Ord$) as is usually done in number theory \cite{DiamondShurman2005} \cite{Vandergeer2012}, i.e., as a $\PSL_2(\Ord)$-orbit of a rational point in $\B_{\QQ} \cup \{ \infty \} \subseteq \HH \cup \{ \infty \}$, we can prove arithmetically the following result concerning to the cusps of the orbifold $M_{\HH ur (\ZZ)}$.

\begin{proposition}\label{1cusp}
$\PSL_2(\HH ur(\ZZ))$ has only one cusp.
\end{proposition}

\begin{proof}
As $\HH ur(\ZZ)$ is a lattice of $\HH$ we have that $\HH ur(\ZZ) \cdot \QQ = \HH$. Then, each element $\beta$ of $\HH$ can be write as $\beta = \alpha c^{-1}$ with $\alpha \in \HH ur(\ZZ)$, $0 \neq c \in \ZZ$. In fact, we can choose $\alpha$ and $c$ relatively prime. 
By right Bézout’s theorem, \cite[Corollary 11.3.6]{Voight2021} there exists $\mu, \nu \in \HH ur(\ZZ)$ such that
\[
\alpha \mu - c \nu = 1.
\]
This gives $\gamma =  \left( \begin{matrix}
\alpha &  \nu \\
c & \mu
\end{matrix}
\right)$ such that $F_\gamma (\infty) =  \alpha c^{-1}$. Using  Lemma 2.4 of \cite{BisiGentili2008} and the commutativity of $c$, we have that 
\[
{\det}_\HH (\gamma) = \sqrt{\vert \alpha \vert ^2 \vert \mu \vert ^2 + \vert c \vert ^2 \vert \nu \vert ^2 - 2 \Re (c \overline{\alpha} \nu \overline{\mu})} = \sqrt{\vert c \nu - c \alpha c^{-1} \mu)  \vert ^2} = \sqrt{\vert  \alpha \mu- c \nu)  \vert ^2} = 1
\] 
then, $\gamma \in \PSL_2(\HH ur(\ZZ))$. 
Thus, the only cusp  of $\PSL_2(\HH ur(\ZZ))$ is the orbit of $\infty$.
\end{proof}

\begin{remark}
In analogy with the class number $h_{\ZZ_K}$ of the ring of integers $\ZZ_K$ of a number field $K$, the notion of a left (resp. right) class number $h_\Ord$ associated with a $\ZZ_K$-order $\Ord$ in a quaternion algebra  $\B$ over $K$ is well understood \cite[Chapter 17]{Voight2021}. In particular, from \cite[Proposition 11.3.4]{Voight2021}, we have that $h_{\HH ur(\ZZ)} = 1$. Then, the previous result is compatible with the classical fact that the number of cusps of a Bianchi \cite[\S 7.2]{ElstrotGrunewaldMennicke1998} or Hilbert-Blumenthal \cite[(1.1)Proposition]{Vandergeer2012} orbifold is equal to the class number of the associated quadratic field.
\end{remark}


\subsection{Hilbert-Blumenthal quaternionic modular group}\label{sec:42}

Let $K=\QQ(\sqrt{n})$ be a real quadratic field, $\ZZ_K$ its ring of integers and $\Ord$ be a $\ZZ_K$-order in the totally definite quaternion algebra $\B_K$. It is well-known that the embedding $\ZZ_K \hookrightarrow \RR$ is not discrete. Then, contrary to the previous case, we cannot discreetly embed $\Ord$ into $\HH$ and consequently $\PSL_2(\Ord)$ is not a discrete subgroup of $\PSL_2(\HH)$. However, $\ZZ_K$ admits a discrete embedding \cite[\S 4.2 Proposition 2]{Samuel70}
\[ 
\ZZ_K \hookrightarrow \RR \times \RR,
\] 
via the Galois twist $w = \alpha + \beta \sqrt{n} \mapsto (w, \sigma(w)) := ( \alpha + \beta \sqrt{n}, \alpha - \beta \sqrt{n})$, which induce a discrete embedding
\begin{equation}\label{embed}
	\Ord \hookrightarrow\HH \times\HH,
\end{equation}
given by $\q  \mapsto	(\q, \sigma(\q)) := \big( x_0 + x_1 i + x_2 j + x_3 k, \; \sigma (x_0) + \sigma (x_1) i + \sigma (x_2) j + \sigma (x_3) k   \big)$, and finally (\ref{embed}) extend to a discrete embedding 
\begin{equation}\label{embed1}
\PSL_2(\Ord) \hookrightarrow \PSL_2(\HH) \times \PSL_2(\HH).
\end{equation}
Then, we can identify the elements of $\PSL_2(\Ord)$ with their image under (\ref{embed1}). 
Observe that the embedding (\ref{embed}) is more natural from the point of view of Minkowski's geometry of numbers in the sense that as $\B_K$ is a totally definite quaternion algebra over a real quadratic field then $\B_K \otimes_\QQ  \RR \cong \HH \times \HH$.

Let $\Hi_{\HH}^2 := \Hi_{\HH}^1 \times \Hi_{\HH}^1 \subseteq \HH \times \HH$, which is isometric to $\Hi_{\RR}^4 \times \Hi_{\RR}^4$ with the Riemannian product metric of Poincare metrics. As expected, $\PSL_2(\HH) \times \PSL_2(\HH)$ acts on $\HH \times \HH$ by quaternionic Möbius transformations $(\gamma_1, \gamma_2) \cdot (\q_1, \q_2) = (F_{\gamma_1} (\q_1), F_{\gamma_2} (\q_2))$ but  not all $(\gamma_1, \gamma_2) \in \PSL_2(\HH) \times \PSL_2(\HH)$ leaves invariant $\Hi_{\HH}^2$.
As in the 1-dimensional case, the set $\M_{\Hi_{\HH}^2} \subseteq \PSL_2(\HH) \times \PSL_2(\HH)$ of couples $(\gamma_1, \gamma_2)$ of matrices that leaves invariant $\Hi_{\HH}^2$ is isomorphic to $Conf_+(\Hi_{\HH}^2)$ and $Isom_+(\Hi_{\HH}^2)$, and it can be characterized as the set of $(\gamma_1, \gamma_2) \in \PSL_2(\HH) \times \PSL_2(\HH)$ such that both $\gamma_1$ and $\gamma_2$ satisfy the BG-conditions. Moreover, $\M_{\Hi_{\HH}^2}$ acts on the boundary $\partial \Hi_{\HH}^2 \cong \mathbb{S}^3 \times \mathbb{S}^3$ of $\Hi_{\HH}^2$ and $\M_{\Hi_{\HH}^2} \cong Conf_+(\mathbb{S}^3 \times \mathbb{S}^3)$.
 
Now, we are ready to describe a special kind of  isometries of $\Hi_{\HH}^2$ lying in 
\[
\PSLBG(\Ord) := \PSL_2(\Ord) \cap \M_{\Hi_{\HH}^2},
\]
 which will be used to define our quaternionic modular group. 

\begin{definition}
Let $\Ord$ be a $\ZZ_K$-order in $\B_K$ and $\Im \Ord$ be the set of pure elements of $\B_K$ lying in $\Ord$. We define the \emph{subgroup of $\Im \Ord$-translations} of $\PSLBG(\Ord)$ by elements of $\Im \Ord$ as 
\[
\T_{\Im \Ord}  := \left\lbrace T_b = \begin{pmatrix}
1 & b \\
0 & 1
\end{pmatrix} \in \PSL_2(\Ord) \mid  \; b \in \Im \Ord \right\rbrace .
\]
\end{definition}

A \emph{translation} in $\Hi_{\HH}^2$ is defined as a transformation 
\[
T_{(b_1, b_2)} := (F_{\gamma_1}, F_{\gamma_2}): \Hi_{\HH}^2 \longrightarrow \Hi_{\HH}^2
\] 
associated to a couple of matrices of the form 
\[
(\gamma_1, \gamma_2)= \left( \begin{pmatrix}
1 & b_1 \\
0 & 1
\end{pmatrix},
\begin{pmatrix}
1 & b_2 \\
0 & 1
\end{pmatrix} \right) \in \M_{\Hi_{\HH}^2},
\] 
where $b_1$ and $b_2$ are such that $\Re(b_1) =0$ and $\Re(b_2)=0$. Note that, if $b \in \Ord$ and $\Re(b) =0$, then $\Re(\sigma(b)) =0$ and we can identify the group $\mathcal{T}_{\Im \Ord}$ with the set of translations in $\Hi_{\HH}^2$ of the form $T_{(b, \sigma(b))}$, with $b \in \Im \Ord$.

\begin{remark}
Note that, if $\Ord$ is the $\ZZ_K$-order $\HH(\ZZ_K) := \{ \q = t + x i + y j + zk \mid w,x,y,z \in \ZZ_K \}$ of $\B_K$, we have that 
\[
\Im \HH(\ZZ_K) = \left\lbrace \frac{1}{2}(\q-\overline{\q}) \mid \q \in \HH(\ZZ_K) \} = \{ x i + y j + zk : x,y,z \in \ZZ_K \right\rbrace,
\] 
in analogy with the imaginary part of a complex number.
However, $\Im \Ord$ is not always equal to the set  $\{\frac{1}{2}(\q-\overline{\q}) \mid \q \in \Ord \}$. For example, let $K=\QQ(\sqrt{n})$, with $n \notequiv 1 \mod 4$ (then its ring of integers $\ZZ_K = \ZZ[\sqrt{n}]$), and consider the $\ZZ[\sqrt{n}]$-order of $\B_{\QQ(\sqrt{n})}$
\[
\Hur(\ZZ[\sqrt{n}]) = \left\lbrace \q = t+xi+yj+zk \mid t,x,y,z \in \ZZ[\sqrt{n}] \mbox{ or } t,x,y,z \in \ZZ[\sqrt{n}] +\frac{1}{2} \right\rbrace .
\] 
It is easy to see that $\frac{i+j+k}{2}  \in \{\frac{1}{2}(\q-\overline{\q}) \mid \q \in \ZZ[\sqrt{n}] \} \subseteq \B_K^0$, by taking $\q = \frac{1+i+j+k}{2} \in \ZZ[\sqrt{n}]$, but $\frac{i+j+k}{2} \notin \Hur(\ZZ[\sqrt{n}])$ and in particular $\frac{i+j+k}{2} \notin \Im \Hur(\ZZ[\sqrt{n}])$.
\end{remark}

\begin{definition}\label{scalu} 
Let $\Ord$ be a $\ZZ_K$-order in $\B_K$ and $\varepsilon$ be the fundamental unit of $\ZZ_K$. We define the \emph{scalar unitary subgroup} of $\PSLBG(\Ord)$ as the set of matrices 
\[
\U_\varepsilon(\Ord) := \left\lbrace
D_\ell := \begin{pmatrix}
\varepsilon^\ell & 0 \\
0 & \varepsilon^{-\ell}
\end{pmatrix} \in \PSL_2(\Ord) \mid \ell \in \NN \right\rbrace .
\]
\end{definition}

A \emph{left bi-homothetic transformation} in $\Hi_{\HH}^2$ is defined as a transformation $h_{(c_1, c_2)} : \Hi_{\HH}^2 \rightarrow \Hi_{\HH}^2$ given by the map $(\q_1,\q_2) \mapsto (c_1 \q_1, c_2 \q_2)$, where $c_1, c_2 \in \HH$ are such that $\Re(c_1 \q_1) >0$ and $\Re(c_2 \q_2) >0$. Note that $D_\ell \in \U_\varepsilon(\Ord)$ defines the left bi-homothetic transformation $h_{(\varepsilon^{2\ell} ,  \sigma(\varepsilon)^{2\ell})}$ in $\Hi_{\HH}^2$.

\begin{definition}\label{torsu} 
Let $\Ord$ be a $\ZZ_K$-order in $\B_K$. We define the \emph{torsion unitary subgroup} of $\PSLBG(\Ord)$ as the set of matrices 
\[
\U^1(\Ord) := \left\lbrace
D_\uu := \begin{pmatrix}
\uu & 0 \\
0 & \uu
\end{pmatrix} \in \PSL_2(\Ord) \mid \uu \in \Ord^1 \right\rbrace .
\]
\end{definition}

Recall from \S \ref{sec:21}, that $\HH^1$ acts by rotation on $\HH^0 \simeq \RR^3$ via conjugation. Then, we define a \emph{left bi-rotation} in $\Hi_{\HH}^2$ as a transformation $r_{(\uu_1, \uu_2)} :  \Hi_{\HH}^2 \rightarrow \Hi_{\HH}^2$ given by the map $(\q_1,\q_2) \mapsto (\uu_1 \q_1 \uu_1^{-1}, \uu_2 \q_2 \uu_2^{-1})$, where $\uu_1, \uu_2 \in \HH^1$. Note that as $\Ord^1 \subseteq \HH^1$ and $\sigma(\uu) \in \Ord^1$ for all $\uu \in \Ord^1$, then $D_\uu \in \U^1(\Ord)$ defines the left bi-rotation $r_{(\uu,  \sigma(\uu))}$ in $\Hi_{\HH}^2$.

Finally the \emph{inversion} in $\Hi_{\HH}^2$ is defined as $(I,I): \Hi_{\HH}^2 \rightarrow \Hi_{\HH}^2$, were $I$ is the usual inversion defined by the matrix $ \begin{pmatrix}
0 & 1 \\
1 & 0
\end{pmatrix}$.

\begin{definition}\label{ModGrp}
Let $K$ be a real quadratic field and $\Ord$ be a $\ZZ_K$-order in the quaternion algebra $\B_K$. We define the \emph{Hilbert-Blumenthal quaternionic modular group} $\Gamma(\Ord) \subseteq \PSLBG(\Ord)$ associated  to $\Ord$ as the group generated by $\U_\varepsilon(\Ord)$, $\U^1(\Ord)$, $\T_{\Im \Ord}$ and $I$. 
\end{definition}

\begin{remark}
When $N_{K/ \QQ} (\varepsilon) = 1$, $\Ord^\times$ could be a degree two extension of $\ZZ_K^\times \Ord^1$ and, in such case, $\Ord^\times \simeq \ZZ_K^\times \Ord^1 \langle 1+ i \rangle$ which follows from \cite[Proposition 6]{Vigneras1976} and \cite[Table 4.3, \S 8]{Li-Xue-Yu2021}. However, we do not include $1+i$ in Definition \ref{scalu} or in Definition \ref{torsu} (then in Definition \ref{ModGrp}) because in both cases $1+i$ does not produce a matrix satisfying BG-conditions.
\end{remark}

\begin{definition}
Let $K$ be a real quadratic field and $\Ord$ be a $\ZZ_K$-order in the quaternion algebra $\B_K$. We define the \emph{Hilbert-Blumenthal quaternionic orbifold} associated to $\Ord$ as
\[
M_{\Gamma(\Ord)} : = \Gamma(\Ord) \backslash \Hi^2_{\HH}.
\]
\end{definition}


\section{Cusp shapes of Hilbert-Blumenthal quaternionic orbifolds}\label{sec:4}

A first step to understanding the geometry of the Hilbert-Blumenthal quaternionic orbifolds is to be able to describe the shape of its cusps.
In this section, we give a description of the cusp at $\infty$ of $M_{\Gamma (\Ord)}$ following \cite{Hirzebruch1973}, \cite{Mcreynolds2008} and \cite{QuinnVerjovsky2020}.

In order to do that, recall that an $(m,1)$-\emph{torus bundle} is the total space of a fiber bundle with base manifold the circle $\mathbb{S}^1$ and fiber the $m$-torus $\mathbb{T}^m$. 
Then, we say that a topological space $M$ is a \emph{virtual} $(m,1)$-\emph{torus bundle} if $M$ is finitely covered by an $(m,1)$-torus bundle.

Interesting examples of (2,1)-torus bundles can be constructed by using the totally positive units of real quadratic fields as in \cite[\S 2]{Hirzebruch1973}. Let $K=\QQ(\sqrt{n})$ be a real quadratic field with ring of integers $\ZZ_K = \ZZ \oplus \ZZ \theta $ and unit group $\ZZ_K^{\times} =  \{ \pm \varepsilon ^{\ell} : \ell \in \ZZ \}$. For simplicity, let's assume that $N_{K/\QQ}(\varepsilon) = -1$, then $\ZZ_{K +}^\times = \langle \epsilon \rangle $ with $\epsilon = \varepsilon^2$.  As we have seen in $\S$ \ref{sec:42}, $\ZZ_K$ admits a discrete embedding $\ZZ_K \hookrightarrow \RR \times \RR$, via the Galois twist $w  \mapsto (w, \sigma(w))$, producing a lattice $\Lambda$ in $\RR^2$. Thus, considering the action of $\ZZ_{K+}^\times$ on $\ZZ_K$ by multiplication, we can extend this action to an action of $\langle \epsilon \rangle$ on $\Lambda$ by 
\begin{equation}\label{actio}
\varphi_\epsilon (w, \sigma(w)) := (\epsilon w, \sigma(\epsilon) \sigma(w)) = (\epsilon w, \epsilon^{-1} \sigma(w)) 
\end{equation}
and define the semi-direct product
\begin{equation}\label{semipro}
S_\epsilon (\ZZ_K , \ZZ_{K+}^\times) = \Lambda \rtimes _{\varphi_\epsilon} \langle \epsilon \rangle.
\end{equation}
As $\epsilon \in \ZZ_{K+}^\times$ is totally positive, by sending $\epsilon \mapsto \log \epsilon$ and $\sigma(\epsilon) \mapsto \log \sigma(\epsilon)$, we can identified $\ZZ_{K+}^\times$ with an additive subgroup of rank 1 in $\RR$. Here, we think on $\RR$ as the hyperplane $\log x + \log \sigma(x) = 0$ in $\RR^2$, which can be carried out in this way because $\epsilon \sigma(\epsilon) = \epsilon \epsilon^{-1} = 1$.  This allows us to extend the action of $\langle \epsilon \rangle$ in $\Lambda$ to an action of $\RR$ on $\RR^2$ and shows that $S_\epsilon (\ZZ_K , \ZZ_{K+}^\times)$ embeds naturally as a discrete subgroup in the solvable Lie group $S (\RR^2 , \RR) = \RR^2 \rtimes \RR$. Then, given the exact sequences (semi-direct products)
\[
0 \longrightarrow \Lambda \longrightarrow  S_\epsilon (\ZZ_K , \ZZ_{K+}^\times) \longrightarrow \langle \epsilon \rangle \longrightarrow 0
\]
\[
0 \longrightarrow \RR^2 \longrightarrow  S(\RR^2 , \RR) \longrightarrow \RR \longrightarrow 0
\]
we have a quotient sequence of coset spaces
\[
0 \longrightarrow \mathbb{T}^2 \longrightarrow M_\epsilon \longrightarrow \mathbb{S}^1 \longrightarrow 0
\]
which proves that the quotient spaces 
\[
M_\epsilon = S_\epsilon (\Lambda, \langle \epsilon \rangle ) \backslash S(\RR^2 , \RR)
\]
is a (2,1)-torus bundle with base $\mathbb{S}^1 = \langle \epsilon \rangle \backslash \RR$ and fiber $\mathbb{T}^2 = \Lambda  \backslash \RR^2$. In particular, $M_\epsilon$ is a solvmanifold with fundamental group 
\[
\pi_1(M_\epsilon) = S_\epsilon (\ZZ_K , \ZZ_{K+}^\times) \cong \ZZ^2 \rtimes \ZZ.
\] 
The last identification is given by mapping  $(\lambda, \epsilon^l) = ((\lambda_1, \lambda_2), \epsilon^l)$ to either $(\lambda_1 = \alpha + \beta \theta, l)$ or $(\lambda_2 = \alpha + \beta \sigma(\theta), -l)$.

\begin{remark}
The action (\ref{actio}) can be described very explicitly if we know the fundamental unit of the real quadratic field. For example, let $K = \QQ(\sqrt{2})$  with ring of integers $\ZZ_K = \ZZ[\sqrt{2}]$ and fundamental unit $\varepsilon = 1 + \sqrt{2}$. By Dirichlet's unit theorem, we have that $\ZZ_K^\times = \{ (1 \pm \sqrt{2})^\ell: \ell \in \ZZ \} $ and, as $N_{K/\QQ}(1+\sqrt{2}) = -1$, $\ZZ_{K+}^\times = \langle (1 + \sqrt{2})^2  \rangle = \{ (1 \pm \sqrt{2})^{2l}: l \in \ZZ \} $. Now, observe that, to describes $(\ref{actio})$ it is enough to describe the action of powers of $\epsilon = (1 + \sqrt{2})^2$ on an integer $\alpha + \beta \sqrt{2} \in \ZZ[\sqrt{2}]$. Then, explicitly we have 
\[
(1+\sqrt2)^2( \alpha +\beta \sqrt2)
	=3\alpha+4\beta+(2\alpha+3\beta)\sqrt2,
\]	
which can be also expressed in matrix terms as the product
$\begin{pmatrix}
		3 & 4\\
		2 & 3
	\end{pmatrix} \begin{pmatrix}
			\alpha \\
			\beta
		\end{pmatrix}$ 
via the identification  $\alpha + \beta \sqrt2\leftrightarrow (\alpha, \beta)^\intercal$.
\end{remark}

The importance of these kind of torus bundles and the explicitness of their construction is due to the fact that all cusp-cross-sections of a Hilbert-Blumenthal surface is a (2,1)-torus bundle constructed in this way \cite{Hirzebruch1973}. In fact, it allows to classify all the Sol 3-manifolds \cite{Mcreynolds2008}. 
In this same line of thought, the main goal of this section is to prove the following result.

\begin{theorem}\label{Teorema}
Let $K$ be a real quadratic field and $\Ord$ be a $\ZZ_K$-order in the quaternion algebra $\B_K$.  Then, a cross-section of the cusp at $\infty$ of the Hilbert-Blumenthal quaternionic orbifold $M_{\Gamma(\Ord)}$ associated to $\Ord$ is a virtual $(6,1)$-torus bundle. 
\end{theorem}

In order to prove our result, we start by constructing a suitable $(6,1)$-torus bundle by using the description of $\Gamma(\Ord)$ given in \S \ref{sec:42}.

Let $\U_\varepsilon (\Ord) \subseteq \Gamma(\Ord)$ be the scalar unitary subgroup of $\PSLBG (\Ord)$ and $\T_{\Im \Ord} \subseteq \Gamma(\Ord)$ be the subgroup of $\Im \Ord$-translations of $\PSLBG(\Ord)$ by elements of $\Im \Ord$. Note that, there is an action of $\U_\epsilon (\Ord)$ on $\Im \Ord$ by Möbius transformations 
\begin{equation}\label{action}
F_{D_\ell} (b) = \begin{pmatrix}
\varepsilon^\ell & 0 \\
0 & \varepsilon^{-\ell}
\end{pmatrix} (b) = \varepsilon^{2\ell} b
\end{equation} 
or equivalently an action of $\U_\varepsilon (\Ord)$ on $\T_{\Im \Ord}$ by conjugation $D_\ell \cdot T_b \cdot D_{-\ell} = T_{2\ell b}$.
Restricting the embedding (\ref{embed}) to $\Im \Ord \subseteq \Ord$, we obtain a discrete embedding $\Im \Ord \hookrightarrow\HH^0 \times\HH^0$ via the Galois twist $b  \mapsto	(b, \sigma(b)) = \big( x_1 i + x_2 j + x_3 k, \;  \sigma (x_1) i + \sigma (x_2) j + \sigma (x_3) k   \big)$, which gives us a lattice $\Lambda$ in $\HH^0 \times\HH^0 \cong \RR^6$. Then, we can extend $(\ref{action})$ to an action on $\Lambda$ by 
\begin{equation}\label{action2}
\varphi_{\varepsilon} (D_\ell, (b, \sigma(b))) := (F_{D_\ell}(b), F_{\sigma(D_\ell)} (\sigma(b)))
\end{equation}
and define the semi-direct product
\begin{equation}\label{semipro}
S_\varepsilon (\Im \Ord , \U_{\varepsilon} (\Ord)) = \Lambda \rtimes _{\varphi_{\varepsilon}} \U_{\varepsilon} (\Ord) \cong \Im \Ord \rtimes 
\U_{\varepsilon} (\Ord) \cong \T_{\Im \Ord} \rtimes \U_{\varepsilon} (\Ord) \cong \ZZ^6 \rtimes \ZZ.
\end{equation}

As in the classical case, due to $\varepsilon^2 \in \ZZ_{K+}^\times$, we can identified $\U_{\varepsilon} (\Ord)$ with an additive subgroup of rank 1 in $\RR$. This allows us to extend (\ref{action2}) to an action of $\RR$ on $\RR^6$ and to embed $S_\varepsilon (\Im \Ord , \U_{\varepsilon} (\Ord))$ as a discrete subgroup in the solvable Lie group $S (\RR^6 , \RR) = \RR^6 \rtimes \RR \cong (\HH^0 \times \HH^0) \rtimes \RR$. Then, as in the previous example, given the exact sequences 
\[
0 \longrightarrow \Lambda \longrightarrow  S_\varepsilon (\Im \Ord , \U_{\varepsilon} (\Ord)) \longrightarrow \U_{\varepsilon} (\Ord) \longrightarrow 0
\]
\[
0 \longrightarrow \RR^2 \longrightarrow  S(\RR^6 , \RR) \longrightarrow \RR \longrightarrow 0
\]
we have a quotient sequence of coset spaces
\[
0 \longrightarrow \mathbb{T}^6 \longrightarrow M_{\varepsilon}(\Ord) \longrightarrow \mathbb{S}^1 \longrightarrow 0,
\]
where $\mathbb{S}^1 = \U_{\varepsilon} (\Ord) \backslash \RR$ and $\mathbb{T}^6 = \Lambda  \backslash \RR^6$, which proves the following result.

\begin{lemma}\label{LemI}
The quotient space 
\[
M_\varepsilon (\Ord) = S_\varepsilon (\Im \Ord , \U_{\varepsilon} (\Ord)) \backslash S(\RR^6 , \RR)
\] 
is a (6,1)-torus bundle with fundamental group $\pi_1(M_\varepsilon(\Ord)) = S_\varepsilon (\Im \Ord , \U_{\varepsilon} (\Ord)) \cong \ZZ^6 \rtimes \ZZ$.
\end{lemma}

As in the classical case, the action (\ref{action}) of $\U_\varepsilon (\Ord)$ on $\Im \Ord$ can be described very explicitly if we know the fundamental unit of the real quadratic field. For example, let $K=\QQ(\sqrt{n})$ and assume that $n \nequiv 1 \mod 4$. In this case, $\ZZ_K = \ZZ[\sqrt{n}]$ and the fundamental unit of $\ZZ[\sqrt{n}]^\times$ is of the form $\varepsilon = x+y\sqrt{n}$, with $x,y \in \ZZ$. As we know, the action of $\U_\varepsilon (\Ord)$ on $\Im \Ord$ is given by the Möbius action of powers of the matrix
	$\begin{pmatrix}
		x+y\sqrt{n} & 0\\
		0 & (x+y\sqrt{n})^{-1}
	\end{pmatrix}$
on elements of the form $b = x_1 i + x_2 j + x_3 k \in \Im \Ord$. Explicitly, as $x+y\sqrt{n}$ is a real number,  we have that
\[
\begin{pmatrix}
		x+y\sqrt{n} & 0\\
		0 & (x+y\sqrt{n})^{-1}
	\end{pmatrix} ( x_1 i + x_2 j + x_3 k) = (x+y\sqrt{n})^2  x_1 i + (x+y\sqrt{n})^2 x_2 j + (x+y\sqrt{n})^2 x_3 k
\]
Then, the description of the action  can be reduced to describe the action of $\begin{pmatrix}
		x+y\sqrt{n} & 0\\
		0 & (x+y\sqrt{n})^{-1}
	\end{pmatrix}$ on the integers $\alpha + \beta \sqrt{n} \in \ZZ[\sqrt{n}]$ as in the classical case. Since such action is given by
\[
   (x+y\sqrt{n})^2(\alpha+\beta\sqrt{n})
	=(x^2+ny^2)\alpha +2yn \beta+(2y\alpha+(x^2+ny^2)\beta)\sqrt{n},
\]	
it can be expressed as the product $\begin{pmatrix}
		x^2+ny^2 & 2yn\\
		2y & x^2+ny^2
	\end{pmatrix} \begin{pmatrix}
			\alpha \\
			\beta
		\end{pmatrix}$. Thus, we can represent explicitly the action of $\begin{pmatrix}
		x+y\sqrt{n} & 0\\
		0 & (x+y\sqrt{n})^{-1}
	\end{pmatrix}$
on $b = x_1 i + x_2 j + x_3 k \in \Im \Ord$ as the product
\begin{equation}\label{action3}
\begin{pmatrix}
		(x^2+ny^2) & 2yn & 0 & 0 & 0 & 0\\
		2y & (x^2+ny^2) & 0 & 0 & 0 & 0\\
		0 & 0 & (x^2+ny^2) & 2yn & 0 & 0\\
		0 & 0 & 2y & (x^2+ny^2) & 0 & 0\\
		0 & 0 & 0 & 0 & (x^2+ny^2) & 2yn\\
		0 & 0 & 0 & 0 & 2y & (x^2+ny^2)
	\end{pmatrix}
	\begin{pmatrix}
			\alpha_1 \\
			\beta_1 \\
			\alpha_2 \\
			\beta_2 \\
			\alpha_3 \\
			\beta_3
		\end{pmatrix}
\end{equation}
via the identification of $\Im \Ord$ with $\ZZ^6$ given by
\[
	x_1 i + x_2 j + x_3 k = 	( \alpha_1+\beta_1 \sqrt n)i +  (\alpha_2+\beta_2 \sqrt n)j +  (\alpha_3+\beta_3 \sqrt n)k
	\leftrightarrow (\alpha_1, \beta_1, \alpha_2, \beta_2, \alpha_3, \beta_3)^\intercal.
\]		
	
As we will see in the following two examples, when $n \equiv 1 \mod 4$, there is not a general formula as (\ref{action3}) but we can compute the action explicitly case by case. 

\begin{example}
Let $K = \QQ(\sqrt{5})$ with ring of integers $\ZZ_K = \ZZ[\theta] = \ZZ[\frac{1+\sqrt{5}}{2}]$. As in the previous case, the description of the action of $\U_\varepsilon (\Ord)$ on $\Im \Ord$ can be reduced to describe the Möbius action of $\begin{pmatrix}
		\varepsilon & 0\\
		0 & \varepsilon^{-1}
	\end{pmatrix} =
	\begin{pmatrix}
		\frac{1+\sqrt{5}}{2} & 0\\
		0 & \frac{-1+\sqrt{5}}{2}
	\end{pmatrix}$
on the integers $\alpha+\beta \theta = \alpha+\beta \left( \frac{1+\sqrt{5}}{2} \right) \in \ZZ[\frac{1+\sqrt{5}}{2}]$.  Then, we have 
\begin{align*}
 \varepsilon^2 (\alpha  + \beta \theta ) &= \left( \frac{3+\sqrt{5}}{2} \right)\alpha + \left( \frac{3+\sqrt{5}}{2} \right) \beta \theta   \\
&= \alpha + \left( \frac{1+\sqrt{5}}{2} \right) \alpha + \beta \theta + \left( \frac{1+\sqrt{5}}{2} \right) \beta \theta \\
& = \alpha + \left( \frac{1+\sqrt{5}}{2} \right) \alpha +  \left( \frac{1+\sqrt{5}}{2} \right) \beta + \left( \frac{3 +\sqrt{5}}{2} \right) \beta \\
& = \alpha + \left( \frac{1+\sqrt{5}}{2} \right) \alpha +  \beta + 2\left( \frac{1+\sqrt{5}}{2} \right) \beta \\
&=  (\alpha+ \beta) +  (\alpha+2\beta) \left( \frac{1+\sqrt{5}}{2} \right) ,
\end{align*}
which produce the matrix $A=\begin{pmatrix}
		1 & 1\\
		1 & 2
	\end{pmatrix}$ and we obtain the explicit action of 
	$\begin{pmatrix}
		\frac{1+\sqrt{5}}{2} & 0\\
		0 & \frac{-1+\sqrt{5}}{2}
	\end{pmatrix}$
	on an element of the form $( \alpha_1+\beta_1 \left( \frac{1+\sqrt{5}}{2} \right))i +  (\alpha_2+\beta_2 \left( \frac{1+\sqrt{5}}{2} \right))j +  (\alpha_3+\beta_3 \left( \frac{1+\sqrt{5}}{2} \right) )k
 \in \Im \Ord$ as the product
\[
\begin{pmatrix}
		1 & 1 & 0 & 0 & 0 & 0\\
		1 & 2 & 0 & 0 & 0 & 0\\
		0 & 0 & 1 & 1 & 0 & 0\\
		0 & 0 & 1 & 2 & 0 & 0\\
		0 & 0 & 0 & 0 & 1 & 1\\
		0 & 0 & 0 & 0 & 1 & 2
	\end{pmatrix}
	\begin{pmatrix}
			\alpha_1 \\
			\beta_1 \\
			\alpha_2 \\
			\beta_2 \\
			\alpha_3 \\
			\beta_3
		\end{pmatrix}
\]
\end{example}

\begin{example}
Let $K=\QQ(\sqrt{13})$ with ring of integers $\ZZ_K = \ZZ[\theta] =  \ZZ[\frac{1+\sqrt{13}}{2}]$. The Möbius action of
	$\begin{pmatrix}
		\varepsilon & 0\\
		0 & \varepsilon^{-1}
	\end{pmatrix} =
	\begin{pmatrix}
		\frac{3+\sqrt{13}}{2} & 0\\
		0 & \frac{-3+\sqrt{13}}{2}
	\end{pmatrix}$
on an integer $\alpha+\beta \theta = \alpha +\beta \left( \frac{1+\sqrt{13}}{2} \right) \in  \ZZ[\frac{1+\sqrt{13}}{2}]$ is given by
\begin{align*}
\varepsilon^2 (\alpha + \beta \theta ) & = \varepsilon^2 \alpha + \varepsilon^2 \beta \theta \\
    & = \left( \frac{11+3 \sqrt{13}}{2} \right) \alpha + \left( \frac{11+3\sqrt{13}}{2} \right) \beta \theta  \\
    & = \alpha (4 + 3\theta) + \beta \theta (4+3\theta) \\
    & = 4 \alpha + 3 \alpha \theta + 4 \beta \theta + 3 \beta \theta^2 \\
    & =  4 \alpha + 3 \alpha \theta + 4 \beta \theta + 3 \beta \left( \frac{7 + \sqrt{13}}{2} \right) \\
    & = 4 \alpha + 3 \alpha \theta + 4 \beta \theta + 3 \beta (3+ \theta) \\
    & = (4 \alpha + 9 \beta ) + (3 \alpha + 7 \beta )\theta ,
\end{align*}
which produce the explicit action of the matrix 
$A = \begin{pmatrix}
		\frac{3+\sqrt{13}}{2} & 0\\
		0 & \frac{-3+\sqrt{13}}{2}
	\end{pmatrix}$
	on $( \alpha_1+\beta_1 \left( \frac{1+\sqrt{13}}{2} \right))i +  (\alpha_2+\beta_2 \left( \frac{1+\sqrt{13}}{2} \right))j +  (\alpha_3+\beta_3 \left( \frac{1+\sqrt{13}}{2} \right) )k
 \in \Im \Ord$ given by the product
\[
\begin{pmatrix}
		4 & 9 & 0 & 0 & 0 & 0\\
		3 & 7 & 0 & 0 & 0 & 0\\
		0 & 0 & 4 & 9 & 0 & 0\\
		0 & 0 & 3 & 7 & 0 & 0\\
		0 & 0 & 0 & 0 & 4 & 9\\
		0 & 0 & 0 & 0 & 3 & 7
	\end{pmatrix}
	\begin{pmatrix}
			\alpha_1 \\
			\beta_1 \\
			\alpha_2 \\
			\beta_2 \\
			\alpha_3 \\
			\beta_3
		\end{pmatrix}.
\]
\end{example}

\emph{Proof of Theorem \ref{Teorema}.}
As we see in the first section 
\[
\Stab_{\M_{\Hi_\HH^1}} (\infty) = \mathcal{A}(\HH) \cong \RR^3 \rtimes (\RR^+ \times O(3)). 
\]
Then, the set of translations in $\Gamma(\Ord)$ fixing $\infty$ is given by $\T_{\Im \Ord} \cong \Im \Ord$, the set of homotheties in $\Gamma(\Ord)$ is given by $\U_\varepsilon(\Ord) \cong \ZZ^\times_{K+}$ and the set of rotations in $\Gamma(\Ord)$ fixing 0 and $\infty$ is given by $\U^1(\Ord) \cong \Ord^1$. Hence,
\[
\Stab_{\Gamma(\Ord)} (\infty) = \mathcal{A}(\Ord) := \T_{\Im \Ord} \rtimes (\U_\varepsilon(\Ord) \times \U^1(\Ord)) \cong \Im \Ord \rtimes (\ZZ_{K+}^\times \times \Ord^1). 
\]
and we can define the orbifold $M_{\A(\Ord)} : = \A(\Ord) \backslash \Hi^2_{\HH}$, which coincide with $M_{\Gamma(\Ord)}$  in a small neighborhood of $(\infty, \infty)$. 

On the other hand, considering the subgroup $\Delta (\Ord) :=  \T_{\Im \Ord} \rtimes \U_\varepsilon(\Ord) \cong \Im \Ord \rtimes \ZZ_{K+}^\times$, we can define $M_{\Delta(\Ord)} : = \Delta(\Ord) \backslash \Hi^2_{\HH}$. Observe that,  $M_{\Delta(\Ord)}$ can be identified with $M_\varepsilon (\Ord) \times \RR^+$ where $M_\varepsilon (\Ord)$ is the (6,1)-torus bundle defined in Lemma \ref{LemI}. Thus, as $[\mathcal{A}(\Ord): \U^1(\Ord)]$ is finite, we have that $M_\varepsilon (\Ord)$ gives a finite covering of the cross-section of the cusp at $\infty$ of $M_{\A(\Ord)}$, which proves the result.   \qed

\begin{remark}
Contrary to the orbifolds described in \cite{DiazVerjovskyVlacci2015} a complete description of the Hilbert-Blumenthal quaternion orbifolds is still far from our reach. The main reason is the complicated action of units of $\Ord^\times$ in $\Hi_{\HH}^2$. On one hand, the finite torsion unitary subgroup $\U^1(\Ord)$ acts by rotations as in the case of classical Bianchi modular orbifolds \cite[\S 7.3]{ElstrotGrunewaldMennicke1998} or in the case of Lipschitz and Hurwitz modular orbifolds studied in \cite{DiazVerjovskyVlacci2015}. On the other hand, the scalar unitary subgroup $\U_\varepsilon(\Ord)$ acts by homotheties as in the case of classical Hilbert-Blumenthal surfaces \cite{QuinnVerjovsky2020} \cite{Vandergeer2012}. Then, we find in Hilbert-Blumenthal quaternion orbifolds a kind of mixture of classical Bianchi orbifolds and Hilbert-Blumenthal surfaces. Therefore, a complete description of $M_{\Gamma(\Ord)}$ should include dealing with a mixture of the usual complications coming from both types of geometric objects.
\end{remark}




\begin{thebibliography}{80}

\addcontentsline{toc}{section}{Referencias}

\bibitem{Ahlfors1982} Lars V. Ahlfors. Moebius transformations in several dimensions. Ordway Professorship Lectures in Mathematics. Minneapolis, Minnesota: University of Minnesota, School of Mathematics. 150 p. \$ 4.45 (1981)., 1981.

\bibitem{Ahlfors1985} Lars V. Ahlfors. \textit{M\"obius Transformations and Clifford Numbers}, pages 65--73. Springer Berlin Heidelberg, Berlin, Heidelberg, 1985.

\bibitem{AlacaWilliams2004} {\c{S}}aban Alaca and Kenneth S. Williams. \textit{Introductory algebraic number theory}. Cambridge: Cambridge University Press, 2004.

\bibitem{Alperin1987} Roger C. Alperin. An elementary account of Selberg's lemma. \textit{Enseign. Math. (2)}, 33:269--273, 1987.

\bibitem{Bianchi1891} Luigi Bianchi. Geometrische darstellung der gruppen linearer substitutionen mit ganzen complexen coefficienten nebst anwendungen auf die zahlentheorie. \textit{Math.Ann.}, 38(3):313--333, 1891.

\bibitem{Bianchi1892} Luigi Bianchi. Sui gruppi di sostituzioni lineari con coefficienti appartenenti a corpi quadratici immaginari. \textit{Math. Ann.}, 40(3):332--412, 1892.

\bibitem{BisiGentili2008} Cinzia Bisi and Graziano Gentili. Moebius transformations and the Poincar\'e distance in the quaternionic setting. \textit{Indiana Univ. Math. J.}, 58(6):2729--2764, 2009.
  
\bibitem{Blumenthal1903}  O. Blumenthal. \"Uber Modulfunktionen von mehreren Ver\"anderlichen. (Erste H\"alfte.). \textit{Math. Ann.}, 56:509--548, 1903.

\bibitem{Blumenthal1904} O. Blumenthal. On modular forms of several variables. (Second half.). \textit{Math. Ann.}, 58:497--527, 1904.

\bibitem{DiamondShurman2005} Fred Diamond and Jerry Shurman. \textit{A first course in modular forms}, volume 228 of \textit{Grad. Texts Math.} Berlin: Springer, 2005.

\bibitem{DiazVerjovskyVlacci2015} Juan Pablo D\'iaz, Alberto Verjovsky, and Fabio Vlacci. Quaternionic Kleinian modular groups and arithmetic hyperbolic orbifolds over the quaternions. \textit{Geom. Dedicata}, 192:127--155, 2018.

\bibitem{ElstrotGrunewaldMennicke1998} J. Elstrodt, F. Grunewald, and J. Mennicke. \textit{Groups acting on hyperbolic space. Harmonic analysis and number theory}. Springer Monogr. Math. Berlin: Springer, 1998.

\bibitem{Hirzebruch1973} Friedrich E. P. Hirzebruch. Hilbert modular surfaces. \textit{Enseign. Math. (2)}, 19:183--281, 1973.

\bibitem{Li-Xue-Yu2021} Qun Li, Jiangwei Xue, and Chia-Fu Yu. Unit groups of maximal orders in totally definite quaternion algebras over real quadratic fields. \textit{Trans. Am. Math. Soc.}, 374(8):5349--5403, 2021.

\bibitem{Mcreynolds2008} DB McReynolds. Cusps of hilbert modular varieties. In \textit{Mathematical Proceedings of the Cambridge Philosophical Society}, volume 144, pages 749--759. Cambridge Univ Press, 2008.

\bibitem{Neukirch1999} J\"urgen Neukirch. \textit{Algebraic number theory. Transl. from the German by Norbert Schappacher}, volume 322 of \textit{Grundlehren Math. Wiss.} Berlin: Springer, 1999.

\bibitem{QuinnVerjovsky2020} Joseph Quinn and Alberto Verjovsky. Cusp shapes of Hilbert-Blumenthal surfaces. \textit{Geom. Dedicata}, 206:27--42, 2020.

\bibitem{Samuel70} Pierre Samuel. \textit{Th\'eorie alg\'ebrique des nombres. (Deuxième et troisième cycles)}. Paris: Hermann, 2003.

\bibitem{Selberg1960} Atle Selberg. On discontinuous groups in higher-dimensional symmetric spaces. Contrib. Function Theory, Int. Colloqu. Bombay, Jan. 1960, 147--164 (1960), 1960.

\bibitem{Sheydvasser2021} Arseniy (Senia) Sheydvasser. The twisted Euclidean algorithm: applications to number theory and geometry. \textit{J. Algebra}, 569:823--855, 2021.

\bibitem{Vahlen1902}  K. Th. Vahlen. \"Uber Bewegungen und komplexe Zahlen. \textit{Math. Ann.}, 55:585--593, 1902.

\bibitem{Vandergeer2012} Gerard Van Der Geer. \textit{Hilbert modular surfaces}, volume 16. Springer Science \& Business Media, 2012.

\bibitem{Vigneras1976} Marie-France Vign\'eras. Simplification pour les ordres des corps de quaternions totalement définis. \textit{J. Reine Angew. Math.}, 286/287:257--277, 1976.

\bibitem{Vigneras1980} Marie-France Vign\'eras. \textit{Arithm\'etique des alg\'ebres de quaternions}, volume 800 of \textit{Lect. Notes Math.} Springer, Cham, 1980.

\bibitem{Voight2021} John Voight. \textit{Quaternion algebras}, volume 288 of \textit{Grad. Texts Math}. Cham: Springer, 2021.

\end{thebibliography}
\end{document}